\documentclass[12pt,leqno]{amsart}
\usepackage{latexsym,amssymb,amscd,eqname}
\def\NZQ{\Bbb}               
\def\NN{{\NZQ N}}

\def\Mc{{\mathcal{M}}}
\def\Lc{{\mathcal{L}}}
\def\Tc{{\mathcal{T}}}
\def\Rc{{\mathcal{R}}}
\def\B'c{{\mathcal{B'}}}
\def\U'c{{\mathcal{U'}}}

\def\ab{{\bold a}}
\def\bb{{\bold b}}

\def\qb{{\bold q}}
\def\rb{{\bold r}}

\def\opn#1#2{\def#1{\operatorname{#2}}} 
\opn\chara{char}
\opn\length{\ell}
\opn\projdim{proj\,dim}
\opn\injdim{inj\,dim}
\opn\ini{in}
\opn\rank{rank}
\opn\depth{depth}
\opn\height{ht}
\opn\embdim{emb\,dim}
\opn\codim{codim}

\opn\Tr{Tr}
\opn\bigrank{big\,rank}
\opn\superheight{superheight}\opn\lcm{lcm}
\opn\trdeg{tr\,deg}%
\opn\reg{reg}
\opn\lreg{lreg}
\opn\set{set}
\opn\supp{Supp}
\opn\shad{Shad}
\opn\div{div}
\opn\Div{Div}
\opn\cl{cl}
\opn\Cl{Cl}
\opn\Spec{Spec}
\opn\Supp{Supp}
\opn\supp{supp}
\opn\Sing{Sing}
\opn\Ass{Ass}
\opn\Ann{Ann}
\opn\Rad{Rad}
\opn\Soc{Soc}
\opn\Ker{Ker}
\opn\Coker{Coker}
\opn\Im{Im}
\opn\Hom{Hom}
\opn\Tor{Tor}
\opn\Ext{Ext}
\opn\End{End}
\opn\Aut{Aut}
\opn\id{id}

\opn\nat{nat}
\opn\GL{GL}
\opn\SL{SL}
\opn\mod{mod}
\opn\ord{ord}
\opn\aff{aff}
\opn\con{conv}
\opn\relint{relint}
\opn\st{st}
\opn\lk{lk}
\opn\cn{cn}
\opn\core{core}
\opn\vol{vol}
\opn\lex{lex}
\opn\gr{gr}

\def\pot#1#2{#1[\kern-0.28ex[#2]\kern-0.28ex]}

\opn\dirlim{\underrightarrow{\lim}}
\opn\invlim{\underleftarrow{\lim}}
\def\pnt{{\raise0.5mm\hbox{\large\bf.}}}

\let\to=\rightarrow

\def\Implies{\ifmmode\Longrightarrow \else
     \unskip${}\Longrightarrow{}$\ignorespaces\fi}
\def\implies{\ifmmode\Rightarrow \else
     \unskip${}\Rightarrow{}$\ignorespaces\fi}
\def\iff{\ifmmode\Longleftrightarrow \else
     \unskip${}\Longleftrightarrow{}$\ignorespaces\fi}

\let\:=\colon
\newtheorem{Theorem}{Theorem}[section]
\newtheorem{Lemma}[Theorem]{Lemma}
\newtheorem{Corollary}[Theorem]{Corollary}
\newtheorem{Proposition}[Theorem]{Proposition}
\newtheorem{Remark}[Theorem]{Remark}

\newtheorem{Example}[Theorem]{Example}

\newtheorem{Definition}[Theorem]{Definition}

\newtheorem{Theorem1}{Theorem}
\let\epsilon=\varepsilon
\let\phi=\varphi
\let\kappa=\varkappa
\numberwithin{equation}{section}

\title{Powers of lexsegment ideals with linear resolution}
\author{Viviana Ene \and Anda Olteanu }
\thanks{The second author was supported by the CNCSIS-UEFISCSU project PN II-RU PD 23/06.08.2010 and by the strategic grant POSDRU/89/1.5/S/58852, Project ``Postdoctoral program for training scientific researchers" co-financed by the European Social Fund within the Sectorial Operational Program Human Resources Development 2007 - 2013"}

\address{Faculty of Mathematics and Computer Science, Ovidius University, Bd.\ Mamaia 124,
 900527 Constanta, Romania,} \email{vivian@univ-ovidius.ro} 
\address{Faculty of Mathematics and Computer Science, Ovidius University, Bd.\ Mamaia 124,
 900527 Constanta, Romania,} \email{olteanuandageorgiana@gmail.com}

\begin{document}

\maketitle

\begin{abstract}  All powers of lexsegment ideals with linear resolution (equivalently, with linear quotients) have linear quotients with respect to 
suitable orders of the minimal monomial generators. For a large subclass of the lexsegment ideals the corresponding Rees algebra has a quadratic Gr\"obner basis, thus it is Koszul. We also find other classes of monomial ideals with linear quotients whose powers have linear quotients too.\\



\end{abstract}

\section*{Introduction}

Let $S=K[x_1,\ldots,x_n]$ be the polynomial ring in $n$ variables over a field $K.$ For an integer $d\geq 2,$ we denote by $\Mc_d$ the set of 
all the monomials of $S$ of degree $d.$ A \textit{lexsegment ideal} of $S$ is a monomial ideal generated by a \textit{lexsegment set}, that is a set 
of the form $L(u,v)=\{w\in \Mc_d : u\geq_{\lex}w \geq_{\lex}v\}$ where $u\geq_{\lex}v$ are two given monomials of $\Mc_d.$ 

Lexsegment ideals were introduced in \cite{HM}. Their  homological properties and invariants have been studied in several papers. We refer the reader 
to \cite{ADH}, \cite{BEOS}, \cite{D}, \cite{DH}, \cite{EOS}, \cite{I}, \cite{O}.

In \cite{ADH}, lexsegment ideals with linear resolution are characterized in numerical terms on the ends of the generating lexsegment set. In 
\cite{EOS} it is shown that, for a lexsegment ideal, having a linear resolution is equivalent to having linear quotients with respect to a suitable 
order of the elements in the generating lexsegment set. There are known examples \cite{C} which show that, in general, powers of monomial ideals with linear 
quotients may have no longer linear quotients, or even more, they do not have a linear resolution.

In this paper we show that the lexsegment ideals have a nice behavior with respect to taking powers, namely all powers of a lexsegment ideal with linear 
quotients (equivalently, with linear resolution) have linear quotients too (Theorem~\ref{thmcomp} and Corollary~\ref{cornoncomp}). Therefore, by collecting all the known results, we may now state the 
following

\begin{Theorem1}
Let $u=x_1^{a_1}\cdots x_{n}^{a_n}$ with $a_1>0$ and $v=x_1^{b_1}\cdots x_n^{b_n}$ be monomials of degree $d$ with $u\geq_{lex}v$ and let $I=(L(u,v))$ 
be a  lexsegment ideal. Then the following statements are equivalent;
\begin{itemize}
\item [(1)] $I$ has a linear resolution.
\item [(2)] $I$ has linear quotients.
\item [(3)] All the powers of $I$ have linear quotients.
\item [(4)] All the powers of $I$ have a linear resolution.
\end{itemize}
\end{Theorem1}

In order to prove  $(2)\Rightarrow (3)$ in the above theorem, we are going to study in the first place (Section~\ref{CLI}) the completely lexsegment ideals, that is, those 
whose generating lexsegment set has the property that its shadows are again lexsegment sets, and, secondly (Section~\ref{NCLI}), those  which 
are not completely lexsegment ideals. For the first class of ideals we need to use and develop some of the techniques introduced in \cite{D}. For the second class, we extend some results of \cite{HHV}.

It will turn out that the Rees algebras of the lexsegment ideals which are not completely  have  quadratic Gr\"obner bases,
therefore they are  Koszul (Corollary~\ref{lexKoszul}). For showing this property  we need to slightly extend the notion of  $\ell$-exchange property which was 
defined in \cite{HHV} to the notion of $\sigma$-exchange property. By exploiting this extension, we show in the last section that one may find larger
 classes of monomial ideals for which the Gr\"obner basis of the relation ideal of the Rees algebra $\mathcal{R}(I)$ can be determined (Theorem~\ref{GB}). Moreover, any monomial ideal $I\subset S$ whose minimal monomial generating set satisfies a $\sigma$-exchange property is of fiber type, that is the relations of its Rees algebra $\mathcal{R}(I)$ consist of the relations of the symmetric algebra $\mathcal{S}(I)$ and of  the fiber relations (Corollary~\ref{corfiber}). We also show that the equigenerated monomial ideals whose minimal monomial generating set satisfies a $\sigma$-exchange property have the nice property that all their powers have linear quotients (Theorem~\ref{powersexchange}).

\section{Preliminaries}

In this section we recall the basic definitions and known results needed for the other sections.

Let $K$ be a field and $S=K[x_1,\ldots,x_n]$ the polynomial ring in $n$ variables over $K.$ For an integer $d\geq 2,$ we denote by 
$\Mc_d$ the set of the monomials of degree $d$ in $S$ ordered lexicographically with $x_1> x_2>\cdots > x_n.$ For two monomials 
$u,v\in \Mc_d$ such that $u\geq_{\lex} v,$ we denote by $L(u,v)$ the lexsegment set bounded by $u$ and $v,$ that is,
\[
L(u,v)=\{w\in\Mc_d \colon u\geq_{\lex} w\geq_{\lex} v\}.
\]

If $u=x_1^d,$ then $L(u,v)$ is denoted $L^i(v)$ and is called the \textit{initial lexsegment} determined  by $v.$ Similarly, if 
$v=x_n^d,$ then $L(u,v)$ is denoted by $L^f(u)$ and is called the \textit{final lexsegment} determined by $u.$ An (\textit{initial, 
final}) \textit{lexsegment ideal} of $S$ is a monomial ideal generated by an (initial, final) lexsegment set. According to \cite{D}, 
we denote by $\Lc_{u,v}$ the $K$-subalgebra of $S$ generated by the monomials of $L(u,v).$ In \cite{D} it is proved that $\Lc_{u,v}$
is a Koszul algebra. More precisely, it is shown that the presentation ideal of $\Lc_{u,v}$ has a Gr\"obner basis of quadratic 
binomials. We briefly recall the basic tools used in \cite{D} in proving this result, since they will be also useful in the next 
section. \\

Let $V_{n,d}$ be the Veronese subring of $S,$ that is, $V_{n,d}=K[\Mc_d].$ Let $w$ be a monomial in $\Mc_d.$ One can write 
$w=x_{\ab}=x_{a_1}\cdots x_{a_d},$ where $1\leq a_1\leq \cdots \leq a_d\leq n.$ Consider the set of variables 
\[
\mathbb{T}=\{T_{\ab}\colon \ab=(a_1,\ldots,a_d)\in \NN^d, 1\leq a_1\leq \cdots \leq a_d\leq n\},
\] 
and let $\varphi\colon K[\mathbb{T}]\to V_{n,d}$ be the $K$-algebra homomorphism defined by 
\[
\varphi(T_{\ab})=x_{\ab}=x_{a_1}\cdots  x_{a_d}.\
\]
Then $V_{n,d}\cong K[\mathbb{T}]/\ker \varphi$ and $P=\ker \varphi$ is called the \textit{toric} or the \textit{presentation ideal} of $V_{n,d}.$

If $\ab=(a_1,\ldots,a_d)$ and $\bb=(b_1,\ldots,b_d)$ are vectors with $1\leq a_1\leq \cdots\leq a_d\leq n$ and 
$1\leq b_1\leq \cdots\leq  b_d\leq n,$ we say that $\ab > \bb$ if $x_{\ab}>_{\lex} x_{\bb},$ that is, if there exists 
$s\geq 1$ such that $a_i=b_i$ for $i\leq s-1$ and $a_s < b_s.$ In this way, one gets a total order on the variables of $\mathbb{T}$ 
by setting $T_{\ab}> T_{\bb}$ if $\ab > \bb.$ Let $>_{\lex}$ be the lexicographic order on $K[\mathbb{T}]$ induced by this order of the variables of 
$\mathbb{T}.$ Namely, we have $T_{\ab(1)}\cdots T_{\ab(N)}>_{\lex} T_{\bb(1)}\cdots T_{\bb(N)}$ if there exists 
$1\leq t\leq N$ such that $T_{\ab(i)}=T_{\bb(i)}$ for $i\leq t-1$ and $T_{\ab(t)}> T_{\bb(t)}.$ 

 A \textit{tableau} is an $N\times d$-matrix $A=[\ab(1),\ldots,\ab(N)]$ with entries in $\{1,\ldots,n\}$, with the property that in every row 
$\ab(i)=(a_{i1},\ldots,a_{id})$ we have $a_{i1}\leq\cdots\leq a_{id}$ and the row vectors are in decreasing lexicographic order, that is $
\ab(1)>\ab(2)>\cdots>\ab(N)$ or, equivalently, $T_{\ab(1)}>T_{\ab(2)}>\cdots>T_{\ab(N)}$. The \textit{support} of $A$ is 
the collection $\supp(A)$ of the integers which appear in the tableau with their occurrences. It is clear that one may associate to each tableau $A$ 
its corresponding monomial $T_A:=T_{\ab(1)}\cdots T_{\ab(N)}$ in $K[\mathbb{T}]$. A tableau $A=[\ab(1),\ldots,\ab(N)]$ is \textit{standard} if, for 
every tableau $B=[\bb(1),\ldots,\bb(N)]$ of same support, $B\neq A$, one has
\[T_A=T_{\ab(1)}\cdots T_{\ab(N)}<_{\lex}T_{\bb(1)}\cdots T_{\bb(N)}=T_B.
\]
As follows from \cite[Proposition 2.10]{D}, this is equivalent to saying that for any $1\leq i < j\leq N,$ the quadratic monomial 
$T_{\ab(i)}T_{\ab(j)}$ is standard. In \cite[Lemma 2.9]{D} it is shown that a quadratic monomial $T_{\ab}T_{\bb}$ it is standard if and only if
$\ab=\bb$ or there exists $1\leq i\leq d$ such that $a_1=b_1,\ldots,a_{i-1}=b_{i-1}$, $a_i < b_i,$ and, if $i<d,$ then $b_{i+1}\leq \cdots \leq 
b_d\leq a_{i+1}\leq \cdots \leq a_d.$
If $A$ is a standard tableau, then the monomial $T_{A}=T_{\ab(1)}\cdots T_{\ab(N)}$ is called \textit{standard}. 
Given a set $\mathcal{A}$ of $Nd$ indices in the set $\{1,\ldots,n\}$, then there exists a unique standard tableau $A$ of size $N\times d$ with $
\supp(A)=\mathcal{A}$. 

We recall the recursive procedure given in \cite{D} to construct a standard tableau $A$ with a given support $\mathcal{A}=\{b_1,\ldots,b_{Nd}\}$ where 
$1\leq b_1\leq\cdots\leq b_{Nd}\leq n$. Namely, if $A=[\ab(1),\ldots,\ab(N)]$, where $\ab(i)=(a_{i1},\ldots,a_{id})$ for $1\leq i\leq N$, then we 
proceed as follows. We put $b_1,\ldots,b_N$ on the first column of $A$, that is,
	\[a_{11}=b_1, a_{21}=b_2,\ldots,a_{N1}=b_N.
\]
Now we consider the decomposition of $(b_1,\ldots,b_N)$ in blocks of equal integers and fill in each sub-tableau determined by each block from the bottom to the top in an inductive way. We illustrate this procedure by a concrete example.

Let $N=5$, $d=3$, $n=8$ and 
\[\mathcal{A}=\{1,1,2,3,3,4,4,4,5,5,6,6,6,7,8\}.  
\]
We indicate each step of the standard tableau of support $\mathcal{A}$.
	\[
\begin{array}{c}
1\\
1\\
2\\
3\\
3	
\end{array}\ \longrightarrow\ 
\begin{array}{ccc}
1&&\\
1&&\\
2&&\\
3&4&4\\
3&4&5
\end{array}\ 
\longrightarrow\ 
\begin{array}{ccc}
1&&\\
1&&\\
2&5&6\\
3&4&4\\
3&4&5
\end{array}\ \longrightarrow\ 
\begin{array}{ccc}
1&6&7\\
1&6&8\\
2&5&6\\
3&4&4\\
3&4&5
\end{array}
\]

We have

\begin{Proposition} \cite[Proposition 2.11]{D} The set $\mathcal{G}=\{T_{\qb}T_{\rb}-T_{\ab}T_{\bb}:T_{\ab}T_{\bb}\ \mbox{is a standard}$ $\mbox{ 
monomial and }\supp[\ab,\bb]=\supp[\qb,\rb]\}$ is a Gr\"obner basis of the presentation ideal of $V_{n,d}$ with respect to $<_{\lex}$.
\end{Proposition}

Moreover, in \cite[Lemma 2.12]{D}, it was proved that if $[\ab,\bb]$ is a standard tableau and $[\qb,\rb]$ is a non-standard tableau such that 
$\supp[\ab,\bb]=\supp[\qb,\rb]$, then $\qb>\ab\geq\bb>\rb$. Consequently, if $T_{\ab}T_{\bb}$ is a standard monomial and $T_{\qb}T_{\rb}$ is such that 
$\supp[\ab,\bb]=~\supp[\qb,\rb]$, then $x_{\ab},x_{\bb}\in L(u,v)$ if $x_{\qb}$ and $x_{\rb}$ belong to $L(u,v)$. Therefore, the set
\[\mathcal{G}'=\{T_{\qb}T_{\rb}-T_{\ab}T_{\bb}:T_{\ab}T_{\bb}\ \mbox{is a standard monomial, }\supp[\ab,\bb]=\supp[\qb,\rb],\]\[ 
\mbox{and }x_{\qb},x_{\rb}\in L(u,v)\}.
\]
is a Gr\"obner basis of the presentation ideal $J_{\Lc_{u,v}}$ of the toric ring $\Lc_{u,v}$.

\section{Powers of completely lexsegment ideals with linear resolution}
\label{CLI}

In order to study the powers of the completely lexsegment ideals with linear quotients, we need to prove some preparatory results.

\begin{Definition}\rm Let $w_1,\ldots,w_N$ be monomials in $\Mc_d$, $N\geq2$. We call the product $w_1\cdots w_N$ \textit{standard} if $T_{w_1}\cdots 
T_{w_N}$ is a standard monomial, that is, the corresponding tableau is standard. 
\end{Definition}

\begin{Definition}\rm If $w_1,\ldots,w_N$ are monomials in $\Mc_d$, and $w_1\cdots w_N=w'_1\cdots w'_N$, where $w'_1,\ldots,w'_n\in \Mc_d$ and $w'_1\cdots w'_N$ is a standard product, we call $w'_1\cdots w'_N$ the \textit{standard representation} of $w_1\cdots w_N$.
\end{Definition}

\begin{Remark}\rm Let $u_1,\ldots,u_N\in L(u,v)$, where $L(u,v)\subset {\mathcal M}_d$ is a lexsegment set. If $w_1\cdots w_N$ is a standard 
representation of $u_1\cdots u_N$, then $w_1,\ldots, w_N \in L(u,v)$. Indeed, let us assume that $T_{u_1}\cdots T_{u_N}$ is not a standard monomial, 
that is $T_{u_1}\cdots T_{u_N}\in\ini_{<} (P)$, where $P\subset K[\mathbb{T}]$ is the presentation ideal of $V_{n,d}$. Then there exists 
$1\leq i < j\leq N$ and $v_i, v_j\in L(u,v)$ with $u_i> v_i\geq v_j > u_j$ such that $T_{u_i}T_{u_j}-T_{v_i}T_{v_j}\in {\mathcal G}^\prime$. Note 
that $T_{u_1}\cdots T_{u_N}>_{\lex} T_{u_1}\cdots T_{u_{i-1}}T_{v_i}T_{u_{i+1}}\cdots T_{u_{j-1}}T_{v_j}T_{u_{j+1}}\cdots T_{u_N}.$
If  the product $T_{u_1}\cdots T_{u_{i-1}}T_{v_i}T_{u_{i+1}}\cdots T_{u_{j-1}}T_{v_j}T_{u_{j+1}}\cdots T_{u_N}$ is standard, we finished. Otherwise we 
continue the reduction. After a finite number of steps, we reach a standard product whose factors belong 
to the lexsegment set $L(u,v)$.
\end{Remark}

Before stating the preliminary results, we fix some notations. For a monomial $u$ of $ S,$ we denote by $\nu_i(u)$ the exponent of the variable $x_i$
in $u,$ that is, $\nu_i(u)=\deg_{x_i}(u)$ for all $1\leq i\leq n.$ We denote $\supp(u)=\{i: \nu_i(u)>0\}$ and set $\max(u)=\max \supp(u),$ 
$\min(u)=\min \supp(u).$

\begin{Lemma}\label{x_1^d} Let $w_1\cdots w_N$ be a standard monomial and let $x_1^dw_1\cdots w_N=w'_1\cdots w'_{N+1}$ be the standard representation 
of $x_1^dw_1\cdots w_N$. Then $w'_1\geq_{lex} w_1$.
\end{Lemma}

\begin{proof} We make induction on the number of variables. The case $n=2$ is straightforward. Let
$n>2$. One may assume, by induction on the degree $d$ of the monomials, that $\nu_1(w_N)=0$. If $\nu_1(w_1)=0$, then $x_1\mid w'_1$, hence $w'_1>_{lex} 
w_1$. Let $\nu_1(w_1)=1$. Therefore, there exists $0<s<N$ such that $x_1\mid w_s$ and $x_1\nmid w_{s+1}$. If $s+d\geq N+2$, then we finished, since 
$x_1^2\mid w'_1$ by the construction of the standard monomials, and $\nu_1(w_1)=1$. Now, let us consider $s+d\leq N+1$. Let 
\[q=\min\left(\frac{w_1}{x_1}\cdots\frac{w_s}{x_1}\right) \text{ and } q'=\min\left(\frac{w'_1}{x_1}\cdots\frac{w'_{s+d}}{x_1}\right).
\]
Then, since $w_1\cdots w_N$ and $w_1^\prime\cdots w_{N+1}^\prime$ are standard products, we have $q=\min(w_1/x_1),$ $q^\prime=\min(w^\prime/x_1)$,  
$\max(w_j)\leq q\leq \min(w_i/x_1)$ for all $1\leq i\leq s < j\leq N$, and  $\max(w_j^\prime)\leq q^\prime\leq \min(w_i^\prime/x_1)$ for all $1\leq  
i\leq s+d < j\leq N+1.$
If $q<q'$, then we get
	\[\nu_{1<m\leq q}(w_1\cdots w_N):=\sum_{1<m\leq q}\nu_m(w_1\cdots w_N)\geq \deg(w_{s+1}\cdots w_N)=\]
	\[= (N-s)d>(N+1-s-d)d\geq\nu_{1<m\leq q}(w'_1\cdots w'_{N+1}):=\sum_{1<m\leq q}\nu_m(w'_1\cdots w'_{N+1}),
\]
which is impossible since $\nu_m(w_1\cdots w_N)=\nu_m(w'_1\cdots w'_{N+1})$ for all $m>1$. Therefore, we must have $q\geq q^\prime.$ If $q> q^\prime$, 
then we finished since $w_1^\prime/ x_1>_{\lex}w_1/ x_1,$ whence $w_1^\prime>_{\lex}w_1.$ What is left to consider is the case $q=q^\prime.$ In this 
case we have
	\[\nu_{1<m\leq q}(w_1\cdots w_N)=(N-s)d+\nu_q(w_1)+\cdots+\nu_q(w_s)
\]
and
	\[\nu_{1<m\leq q}(w'_1\cdots w'_{N+1})=(N+1-s-d)d+\nu_q(w'_1)+\cdots+\nu_q(w'_{s+d}).
\]
Since $\nu_{1<m\leq q}(w_1\cdots w_N)=\nu_{1<m\leq q}(w'_1\cdots w'_{N+1})$, we obtain
	\[\nu_q(w'_1)+\cdots+\nu_q(w'_{s+d})=d(d-1)+\nu_q(w_1)+\cdots+\nu_q(w_s).
\]
This implies that
	\[\frac{w'_1}{x_1}\cdots\frac{w'_{s+d}}{x_1}=x_q^{d(d-1)}\left(\frac{w_1}{x_1}\cdots\frac{w_{s}}{x_1}\right).
\]
Note that \large$\frac{w_1}{x_1}\cdots\frac{w_{s}}{x_1}\ $\normalsize is a standard product in the variables $x_q,\ldots,x_n$. Applying induction on the number $n$ of variables, we have, after $d$ steps, that
	\[x_q^{d(d-1)}\left(\frac{w_1}{x_1}\cdots\frac{w_{s}}{x_1}\right)=\bar{w}_1\cdots\bar{w}_{s+d},
\]
where $\bar{w}_1\cdots\bar{w}_{s+d},$ is a standard product and $\bar{w}_1\geq_{lex} w_1/x_1$. But \large$\frac{w'_1}{x_1}\cdots\frac{w'_{s+d}}{x_1}\ $\normalsize is a standard product as well, hence we have $\bar{w}_1= w'_1/x_1\geq_{lex} w_1/x_1$, whence $w'_1\geq_{lex} w_1$.
\end{proof}

\begin{Lemma}\label{x_nu_1} Let $u_1\cdots u_N$ and $w_1\cdots w_N$ be standard products and $u_1\cdots u_N x_n=x_1w_1\cdots w_N$. Then we have $u_1\geq_{lex} w_1.$
\end{Lemma}

\begin{proof} We use induction on $N$. If $N=1$, the inequality $u_1\geq_{lex} w_1$ is obvious. Now we assume $N>1$ and let $u_1\cdots 
u_N=x_{b_1}\cdots x_{b_{Nd}}$, where $1=b_1\leq\cdots\leq b_{Nd}\leq n$ and $\min(u_j)=b_j$ for all $1\leq j\leq N$. We first notice that we may 
assume without loss of generality that $\nu_1(u_i)\leq 1$ for all $1\leq i\leq N.$ If $b_2>b_1$, we obviously have $w_1\leq_{lex} u_1$ since 
$\min(w_1)=b_2$. Therefore, we may assume $b_1=b_2=1$. If $b_1<b_N$, let $k\leq N$ be the largest integer such that $b_{k-1}<b_k=\cdots=b_N$. We have 
$k\geq 3$. Since $u_1\cdots u_N$ is a standard product, we get
	\[u_1\cdots u_{k-1}=x_{b_1}\cdots x_{b_{k-1}}x_{b_{N+(d-1)(N-k+1)+1}}\cdots x_{Nd}.
\]
Similarly, since $w_1\cdots w_{N}$ is a standard product, we get
	\[w_1\cdots w_{k-2}=x_{b_2}\cdots x_{b_{k-1}}x_{b_{N+(d-1)(N-k+2)+2}}\cdots x_{Nd}x_n.
\]
Therefore, there exists a monomial $w\in\Mc_d$, namely
	\[w=x_{b_{N+(d-1)(N-k+1)+1}}\cdots x_{b_{N+(d-1)(N-k+2)+1}}
\]
such that 
\[x_1w_1\cdots w_{k-2}w=u_1\cdots u_{k-1}x_n.
\] 
One observes that $w_1\cdots w_{k-2}w$ and $u_1\cdots u_{k-1}$ are standard products. Then, by induction on $N$, it follows that $w_1\leq_{lex} u_1$.

It remains to consider $b_1=\cdots=b_N=1<b_{N+1}\leq\cdots\leq b_{Nd}$, since, by our assumption on $u_1,\ldots, u_{N}$, we cannot have $b_{N+1}=1$. If 
$b_{N+1}<b_{N+d+1}$, then, by the construction of standard products, we get $w_1<_{lex}u_1$. Let $b_{N+1}=b_{N+2}=\cdots=b_{N+d+1}$. Then we obtain
	\[x_n\cdot\frac{u_1}{x_1}\cdots\frac{u_N}{x_1}=\frac{w_1}{x_1}\cdots\frac{w_{N-1}}{x_1}\left(x_{b_{N+1}}\cdots x_{b_{N+d}}\right),
\]
whence
	\[x_n\left(\frac{u_1}{x_1}\cdots\frac{u_N}{x_1}\right)=x_{b_{N+1}}\left(x_{b_{N+1}}^{d-1}\frac{w_1}{x_1}\cdots\frac{w_{N-1}}{x_1}\right).
\]
Let $w'_1\cdots w'_{N}$ be the standard representation of $x_{b_{N+1}}^{d-1}$\large$\frac{w_1}{x_1}\cdots\frac{w_{N-1}}{x_1}$\normalsize. By Lemma 
\ref{x_1^d}, we have $w'_1\geq_{lex} w_1/x_1$. On the other hand, we have
	\[x_n\left(\frac{u_1}{x_1}\cdots\frac{u_N}{x_1}\right)=x_{b_{N+1}}\left(w'_1\cdots w'_N\right), 
\]
with \large$\frac{u_1}{x_1}\cdots\frac{u_N}{x_1}\ $\normalsize and $w'_1\cdots w'_N$ standard monomials in a number of variables smaller than $n$. By induction on $n$ we get $u_1/x_1\geq_{lex} w'_1$ whence $u_1/x_1\geq_{lex} w_1/x_1$, which yields $u_1\geq_{lex} w_1$.
\end{proof}

\begin{Lemma}\label{N+1} Let $u_1\geq_{lex}\cdots\geq_{lex}u_{N}\geq_{lex}u_{N+1}$ be monomials of degree $d$ with $\nu_1(u_i)\leq 1$ for all 
$1\leq i\leq N,$ such that $u_1\cdots u_N$ is a standard product and $\max(\supp(u_1\cdots u_N))\leq\min(\supp(u_{N+1}))$. Let $v_1\cdots v_{N+1}$ be 
the standard representation of $u_1\cdots u_{N}u_{N+1}$. Then $v_{N+1}\leq_{lex}u_N$.
\end{Lemma}

\begin{proof} We use induction on $N$. For  $N=1$, since $v_1v_2=u_1u_2$ and $v_1v_2$ is a standard product, then we  have 
$u_1>_{lex} v_1\geq_{lex} v_2 >_{lex}u_2$.

Let $N>1$ and assume that $u_1\cdots u_{N}=x_{b_1}\cdots x_{b_{Nd}}$ and $u_{N+1}=x_{b_{Nd+1}}\cdots x_{b_{(N+1)d}}$ with
	\[b_1\leq\cdots\leq b_{Nd}\leq b_{Nd+1}\leq\cdots\leq b_{(N+1)d}.
\]
Since $u_1\cdots u_N$ is a standard product, we have $\min(u_j)=b_j$ for all $1\leq j\leq N$. Since $v_1\cdots v_Nv_{N+1}$ is standard, we have 
$\min(v_j)=b_j$ for all $1\leq j\leq N+1$. If $b_{N+1}>b_N$, we obviously have $v_{N+1}\leq_{lex}u_N$. Therefore, it remains to consider that 
$b_N=b_{N+1}$. Let $1\leq k\leq N$ be the largest integer such that $b_{k-1}<b_k=\cdots=b_N$. We have $k>1$ since otherwise $\nu_1(u_1)\geq 2$. Since 
$u_1\cdots u_N$ is standard, we get that
	\[u_k\cdots u_N=x_{b_k}\cdots x_{b_N}x_{b_{N+1}}\cdots x_{b_{N+(d-1)(N-k+1)}}.
\]
Similarly, since $v_1\cdots v_{N+1}$ is standard, we get
	\[v_k\cdots v_{N+1}=x_{b_k}\cdots x_{b_N}x_{b_{N+1}}\cdots x_{b_{N+(d-1)(N-k+2)+1}}.
\]
Therefore, there exists a monomial $w\in\Mc_d$, namely
	\[w=x_{b_{N+(d-1)(N-k+1)+1}}\cdots x_{b_{N+(d-1)(N-k+2)+1}},
\]
such that $v_k\cdots v_{N+1}=u_k\cdots u_Nw$ and $\max(\supp(u_k\cdots u_N))\leq\min(\supp(w))$. One may note that $u_k\cdots u_N$ and $v_k\cdots 
v_{N+1}$ are  standard products as well. By the induction hypothesis, we get $v_{N+1}\leq_{lex}u_N$.
\end{proof}

\begin{Lemma}\label{uN>wN} Let $u_1,\ldots,u_N,w_1,\ldots,w_N$ be monomials of degree $d$ in $S$ such that $x_nu_1\cdots u_N=x_1w_1\cdots w_N$, where 
$u_1\cdots u_N$, $w_1\cdots w_N$ are standard products. Then $u_N\geq_{lex} w_N$.
\end{Lemma}

\begin{proof} We may assume that $\nu_1(w_N)=0$ which implies that $\nu_1(u_i)\leq 1$ for all $1\leq i\leq N.$ 
Let $u_1\cdots u_N=x_{b_1}\cdots x_{b_{Nd}}$ with $1=b_1\leq\cdots \leq b_{Nd}$ and $\min(u_j)=b_j$ for all $1\leq j\leq N$. We have 
$\min(w_j)=b_{j+1}$ for all $1\leq j\leq N$. If $b_{N+1}>b_N$, then $w_{N}\leq_{lex}u_N$. Let $b_{N+1}=b_N$ and $1\leq k\leq N$ be the 
largest integer such that $b_{k-1}<b_k=\cdots=b_N$. If $k=1,$ then $b_1=\cdots = b_N=b_{N+1}$. Since $w_1\cdots w_N$ is a standard product, we 
get $\nu_1(w_N)> 0,$ which is impossible by our assumption. Therefore, it follows that $k>1.$  Since $u_1\cdots u_{N}$ is a standard 
product, we have 
\[u_k\cdots u_N=x_{b_k}\cdots x_{b_N}x_{b_{N+1}}\cdots x_{b_{N+(d-1)(N-k+1)}}.
\]
Similarly, since $w_1\cdots w_{N}$ is a standard product, we get 
\[w_{k-1}\cdots w_{N}=x_{b_k}\cdots x_{b_N}x_{b_{N+1}}\cdots x_{b_{N+(d-1)(N-k+2)+1}}.
\]
Therefore, if 
\[w=x_{b_{N+(d-1)(N-k+1)+1}}\cdots x_{b_{N+(d-1)(N-k+2)+1}},
\] 
we have 
\[w_{k-1}\cdots w_{N}=u_k\cdots u_{N}w.
\]
and $\max(\supp(u_k\cdots u_N))\leq\min(w)$. Since $u_k\cdots u_N$ and $w_{k-1}\cdots w_{N}$ are also standard products, by using the previous lemma, 
we get $w_{N}\leq_{lex} u_{N}$.
\end{proof}

In order to state the main theorem of this section we need to recall the following

\begin{Theorem} [\cite{EOS},\cite{BEOS}]\label{linear} Let $u=x_1^{a_1}\cdots x_{n}^{a_n}$, with $a_1>0$, and $v=x_1^{b_1}\cdots x_n^{b_n}$ be monomials of degree $d$ with $u\geq_{lex}v$ and let $I=(L(u,v))$ be a completely lexsegment ideal. The following statements are equivalent:
\begin{enumerate}
	\item $u$ and $v$ satisfy one of the following conditions:
	\begin{itemize}
	\item[(i)] $u=x_1^ax_2^{d-a}$, $v=x_1^ax_n^{d-a}$ for some $a$ with $0<a\leq d$;
	\item[(ii)] $b_1<a_1-1$;
	\item[(iii)] $b_1=a_1-1$ and, for the largest monomial $w$ of degree $d$ with $w<_{lex}v$, one has $x_1w/x_{\max(w)}\leq_{lex}u$.
\end{itemize}
\item $I$ has linear quotients.
\item $I$ has a linear resolution.
\end{enumerate}
\end{Theorem}

\begin{Remark}\label{comp}\rm It is obviously that, if a completely lexsegment ideal is determined by $u$ and $v$ satisfying condition (i) in the above theorem, then all its powers have linear quotients. Therefore, we only need to study the powers of completely lexsegment ideals which are determined by monomials $u$ and $v$ satisfying condition (ii) or (iii) in Theorem \ref{linear}.
\end{Remark}

\begin{Theorem} Let $u=x_1^{a_1}\cdots x_{n}^{a_n}$ with $a_1>0$ and $v=x_1^{b_1}\cdots x_n^{b_n}$ be monomials of degree $d$ with $u\geq_{lex}v$ and let $I=(L(u,v))$ be a completely lexsegment ideal with  linear quotients. Then all the powers of $I$ have linear quotients.
\end{Theorem}

\begin{proof} By using Remark \ref{comp}, we have to consider only the cases when $u$ and $v$ satisfy one of the following conditions:
\begin{itemize}
	\item[(a)] $b_1<a_1-1$;
	\item[(b)] $b_1=a_1-1$ and for the largest monomial $w$ of degree $d$ with $w<_{lex}v$, one has $x_1w/x_{\max(w)}\leq_{lex}u$.
\end{itemize}

We recall (see \cite[Theorem 1.2]{EOS}) that in these cases, $I$ has linear quotients with respect to the following  order on $\Mc_d$. For $w,w'\in \Mc_d$ we set $w\succ w'$ if $\nu_1(w)<\nu_1(w')$ or $\nu_1(w)=\nu_1(w')$ and $w>_{lex}w'$. 

Let $N>1$. We show that $I^N$ has linear quotients with respect to the order $\succ$ on the set $\Mc_{Nd}$. Let $u_1\cdots u_N, v_1\cdots v_N\in I^N$ be two standard products such that $v_1\cdots v_N\succ u_1\cdots u_N$. We have to show that there exists a monomial $w\in I^N$ such that $w\succ u_1\cdots u_N$, $w/\gcd(w,u_1\cdots u_N)=x_i$ and $x_i$ divides the monomial $v_1\cdots v_N/\gcd(v_1\cdots v_N,u_1\cdots u_N)$. We have to analyze two cases.

\underline{Case I:} $\nu_1(v_1\cdots v_N)=\nu_1(u_1\cdots u_N)$. By the definition of the order $\succ$, we must have $v_1\cdots v_N>_{lex}u_1\cdots u_N$. Let $i\geq 2$ be the smallest index such that $\nu_i(v_1\cdots v_N)>\nu_i(u_1\cdots u_N)$. We claim that there exists $1\leq q\leq N$ such that $i<\max(u_q)$. Indeed, otherwise we have $i\geq\max(u_1\cdots u_N)$ and obtain
	\[Nd=\deg(u_1\cdots u_N)=\sum_{k=1}^i\nu_k(u_1\cdots u_N)<\sum_{k=1}^i \nu_k(v_1\cdots v_N)\leq Nd,
\]
a contradiction.

Let, therefore, $1\leq q\leq N$ be such that $i<\max(u_q)$. Then we get
	\[\frac{x_iu_q}{x_1}\in L(u,v)\mbox{ or }\frac{x_iu_q}{x_{\max(u_q)}}\in L(u,v)
\]
(see also the proof of \cite[Theorem 1.2]{EOS}).
We recall the  argument which was used in \cite[Theorem 1.2]{EOS} and will be also used in this proof several times. We have $x_iu_q/x_1<_{lex}u_q\leq_{lex}u$ and $x_iu_q/x_{\max(u_q)}>_{lex}u_q\geq_{lex}v$. If we assume that 
$x_iu_q/x_1<_{lex}v$ and $x_1u_q/x_{\max(u_q)}>_{lex}u$, we get $b_1=a_1-1$ and $x_iu_q/x_1\leq_{lex}w$, where $w$ is the largest monomial of degree 
$d$ such that $w<_{lex}v$. We get
	\[\frac{x_iu_q}{x_1x_{\max(x_iu_q/x_1)}}\leq_{lex} \frac{w}{x_{\max(w)}},
\]
which, by using condition (b), leads to
	\[\frac{x_iu_q}{x_{\max(u_q)}}\leq_{lex} \frac{x_1w}{x_{\max(w)}}\leq_{lex}u,
\]
a contradiction. Therefore, one of the monomials $u'_q=x_iu_q/x_1$ or $u''_q=x_iu_q/x_{\max(u_q)}$ belongs to $L(u,v)$. Note that 
$u'_q\succ u_q$ and $u''_q\succ u_q$. Then we may take $w=u_1\cdots u_{q-1}u'_qu_{q+1}\cdots u_N $ or $w=u_1\cdots u_{q-1}u''_qu_{q+1}\cdots u_N$. In 
each case it follows that $w\succ u_1\cdots u_N$, $w/\gcd(w,u_1\cdots u_N)=x_i$ and  $x_i | v_1\cdots v_N/\gcd(v_1\cdots 
v_N,u_1\cdots u_N)$.

\underline{Case II:} $\nu_1(u_1\cdots u_N)>\nu_1(v_1\cdots v_N)$. Then there exist two monomials $m,m'\in S$ of same degree, let us say $p$, such that 
$\gcd(m,m')=1$ and 
\begin{eqnarray}\label{1}
m u_1\cdots u_N=m'v_1\cdots v_N.
\end{eqnarray}

Since $\nu_1(u_1\cdots u_N)>\nu_1(v_1\cdots v_N)$, we get $x_1|m^\prime$ and $x_1\nmid m.$ Let $i=\min(\supp(m)).$ 

If there exists $1\leq q\leq N$ 
such that $i < \max(u_q),$ then, as in the proof of Case (I), we may take $w=u_1\cdots u_q^\prime\cdots u_N$ where $u_q^\prime=x_i u_q/ x_1$ or 
$u_q^\prime= x_i u_q/ x_{\max(u_q)}.$ Then the following conditions hold: $w\succ u_1\cdots u_N$, $w/ \gcd(w,u_1\cdots u_N)=x_i$ and 
$x_i $ divides the monomial $v_1\cdots v_N/ (\gcd(v_1\cdots v_N,u_1\cdots u_N))$. 

Now let $\max(u_q)\leq i$ for all $1\leq q\leq N,$ that is, 
$\supp(u_1\cdots u_N)\subset \{1,\ldots, i\}.$ We show by induction on $p=\deg(m)$ that there exists $j>1$ such that $x_j | m$ and

\begin{eqnarray}\label{2}
x_j u_1\cdots u_N=x_1 w_1\cdots w_N,
\end{eqnarray}
where $w_1,\ldots,w_N\in L(u,v)$ and $w_1\cdots w_N$ is a standard product. If $p=1$, there is nothing to prove. Let $p>1$ and assume that there exists 
$1 < j < i$ such that $x_j | m^\prime.$ There exists $1\leq q\leq N$ such that $j<i\leq \max(v_q)$ since $x_i | v_1\cdots v_N.$ As $j < \max(v_q),$ it follows that 
one of the monomials $x_j v_q/ x_1\in L(u,v)$ or $x_j v_q/ x_{\max(v_q)}\in L(u,v)$. Let us consider that $v_q^\prime=x_j v_q/ x_1\in L(u,v).$
By using (\ref{1}), we get the relation 
\[
m u_1\cdots u_N=(x_1 m^\prime/x_j)(v_1\cdots v^\prime_q\cdots v_N).
\]
If $v_q^{\prime\prime}=x_j v_q/ x_{\max(v_q)}\in L(u,v),$ then, by using again (\ref{1}), we get the relation
\[
m u_1\cdots u_N=(x_{\max(v_q)} m^\prime/x_j)(v_1\cdots v^{\prime\prime}_q \cdots v_N).
\]
These last two relations show that either there exists a relation of the form $m u_1\cdots u_N $ $= x_1^p w_1\cdots w_N$ where $w_1\cdots w_N$ is a standard 
product of monomials of $L(u,v)$, with $\deg(m)=p$ and $x_1\nmid m,$ or we may apply induction on $p$ and reach the desired conclusion. 
In the first case, let $m=x_{i_1}x_{i_2}\cdots x_{i_p}$, with $i=i_1\leq i_2\leq\cdots\leq i_p\leq n$. 
For $j=\overline{1,p}$, let $w_{j1}\cdots w_{jN}$ be the standard product such that
	\[
	x_{i_1}u_1\cdots u_N=x_1w_{11}w_{12}\cdots w_{1N},
\]
\[
	x_{i_2}w_{11}w_{12}\cdots w_{1N}=x_1w_{21}w_{22}\cdots w_{2N},
\]
\[
	x_{i_3}w_{21}w_{22}\cdots w_{2N}=x_1w_{31}w_{32}\cdots w_{3N},
\]
$$\vdots$$
\[
	x_{i_p}w_{p-1,1}w_{p-1,2}\cdots w_{p-1,N}=x_1w_{p1}w_{p2}\cdots w_{pN}.
\]
Multiplying these equalities, we get
	\[
	mu_{1}\cdots u_{N}=x_1^pw_{p1}w_{p2}\cdots w_{pN},
\]
hence $w_{pi}=v_i$, for $1\leq i\leq N$, since $w_{p1}w_{p2}\cdots w_{pN}$ and $v_1\cdots v_N$ are standard products.

It is easily seen that $\supp(w_{j1}\cdots w_{jN})\subset\{1,\ldots,i_j\}$ for all $1\leq j\leq p$. Therefore, we may apply Lemma \ref{x_nu_1} and Lemma \ref{uN>wN} and get
	\[u\geq_{lex}u_1\geq_{lex}w_{11}\geq_{lex}w_{21}\geq_{lex}\cdots\geq_{lex}w_{p1}=v_1\geq_{lex}v
\]
and
\[u\geq_{lex}u_N\geq_{lex}w_{1N}\geq_{lex}w_{2N}\geq_{lex}\cdots\geq_{lex}w_{pN}=v_N\geq_{lex}v.
\]
In particular, we have\[u\geq_{lex}w_{11}\geq_{lex}\cdots\geq_{lex}w_{1N}\geq_{lex}v,
\]
whence $$x_{i_1}u_1\cdots u_N=x_1w_{11}\cdots w_{1N},$$ and $w_{11},\ldots,w_{1N}\in L(u,v)$. Therefore, we have  an equality of the form 
$x_ju_1\cdots u_N=x_1w_1\cdots w_N,$ where $w_1\cdots w_N\in I^N$ is a standard product and $j\geq 2.$ Let $w=(x_ju_1\cdots u_N)/x_1.$ Then 
$w\succ u_1\cdots u_N, w/\gcd(w,u_1\cdots u_N)=x_j$ and $x_j$ divides the monomial $ v_1\cdots v_N/\gcd(v_1\cdots v_N,u_1\cdots u_N),$ which ends our proof.
\end{proof}

Combining the above theorem with \cite[Theorem 1.3]{ADH} and \cite[Theorem 1.2]{EOS}, we get the following equivalent statements.

\begin{Theorem}\label{thmcomp}
Let $u=x_1^{a_1}\cdots x_{n}^{a_n}$ with $a_1>0$ and $v=x_1^{b_1}\cdots x_n^{b_n}$ be monomials of degree $d$ with $u\geq_{lex}v$ and let $I=(L(u,v))$ 
be a completely lexsegment ideal with  linear quotients. The the following statements are equivalent;
\begin{itemize}
\item [(1)] $u$ and $v$ satisfy one of the following conditions:
\begin{itemize}
	\item[(i)] $u=x_1^ax_2^{d-a}$, $v=x_1^ax_n^{d-a}$ for some $a$ with $0<a\leq d$;
	\item[(ii)] $b_1<a_1-1$;
	\item[(iii)] $b_1=a_1-1$ and, for the largest monomial of degree $d$ with $w<_{lex}v$, one has $x_1w/x_{\max(w)}\leq_{lex}u$.
\end{itemize}
\item [(2)] $I$ has a linear resolution.
\item [(3)] $I$ has linear quotients.
\item [(4)] All the powers of $I$ have linear quotients.
\item [(5)] All the powers of $I$ have a linear resolution.
\end{itemize}
\end{Theorem}

\section{Exchange properties and applications}
\label{NCLI}

We first fix some notations. As in the previous section, let $S=K[x_1,\ldots,x_n]$ be the ring of polynomials in $n$ variables over a field $K$ and $\Mc_d$ the set of all monomials of degree $d$ in $S$. If $B\subset\Mc_d$ is a nonempty set, we denote by $K[B]$ the $K$-subalgebra of $S$ generated by the monomials of $B$.

Let $R=K[\{T_u\}_{u\in B}]$ be the polynomial ring in a set of variables indexed over $B$ and $\pi:R\rightarrow K[B]$ the surjective $K$-algebra homomorphism defined by $\pi(T_u)=u$, for all $u\in B$. $J_{K[B]}:=\ker\pi$ is called the \textit{toric ideal} of $K[B]$.

Let $<$ be a monomial order on $R$ and $\ini_{<}(J_{K[B]})$ the initial ideal of $J_{K[B]}$ with respect to $<$. A monomial $T_{u_1}\cdots T_{u_N}\in R$ is a \textit{standard monomial} of $J_{K[B]}$ with respect to $<$ if $T_{u_1}\cdots T_{u_N}\notin\ini_{<}(J_{K[B]})$. We recall the following definition which was given in \cite{HHV}.

\begin{Definition}{\cite[Definition 4.1]{HHV}}\rm$\ $ We say that a nonempty set $B\subset\Mc_d$ satisfies the \textit{$\ell$-exchange property} with respect to a monomial order $<$ on $R$ if $B$ posseses the following property: if $T_{u_1}\cdots T_{u_N}$ and $T_{v_1}\cdots T_{v_N}$ are standard monomials of $J_{K[B]}$ with respect to $<$ such that 
\begin{itemize}
	\item[(a)] $\nu_i(u_1\cdots u_N)=\nu_{i}(v_1\cdots v_N)$ for $1\leq i\leq q-1$ (with $q\leq n-1$),
	\item[(b)] $\nu_q(u_1\cdots u_N)<\nu_{q}(v_1\cdots v_N)$,
\end{itemize}
then there exist $1\leq\delta\leq N$,  and $q<j\leq n$ with  $j\in\supp(u_{\delta})$ and $x_qu_{\delta}/x_j\in B$.
\end{Definition}

Inspired by this definition we consider the following slight generalization.
Let $<_{\sigma}$ be a monomial order on $S$.

\begin{Definition}\rm We say that $B$ satisfies the \textit{$\sigma$-exchange property} with respect to $<$ if $B$ has the following property: if $T_{u_1}\cdots T_{u_N}$ and $T_{v_1}\cdots T_{v_N}$ are standard monomials of $J_{K[B]}$ with respect to $<$ such that $u_1\cdots u_{N}<_{\sigma}v_{1}\cdots v_N$, then there exist $1\leq\delta\leq N$, $q\in\supp(v_1\cdots v_N)$, and $j\in\supp(u_{\delta})$ such that
\begin{itemize}
	\item[(i)] $\nu_q(u_1\cdots u_N)<\nu_{q}(v_1\cdots v_N)$,
	\item[(ii)] $x_j<_{\sigma}x_q$,
	\item[(iii)] $x_qu_{\delta}/x_j\in B$. 
\end{itemize}
\end{Definition}

It is straightforward to show that if $B$ satisfies the $\ell$-exchange property with respect to a monomial order $<$ on $R$, then $B$ satisfies the $\sigma$-exchange property with respect to $<$ for $<_{\sigma}= <_{lex}$ on $S$ with $x_1>_{lex}\cdots>_{lex}x_n$.

\begin{Example}\label{explu}\rm Let $<_{\sigma}$ be a monomial order on $S$ defined as follows. For $m,m'$  monomials in $S$, we set $m<_{\sigma} m'$ if $\deg(m)<\deg(m')$ or $\deg(m)=\deg(m')$ and $m>_{revlex}m'$, that is, if $m=x_1^{a_1}\cdots x_n^{a_n}$, $m'=x_1^{b_1}\cdots x_n^{b_n}$, then there exists some $1\leq s\leq n$ such that $a_n=b_n$, $a_{n-1}=b_{n-1}, \ldots,a_{s+1}=b_{s+1}$, and $a_s<b_s$. In particular, we have 
$x_n >_{\sigma} x_{n-1}>_{\sigma}\cdots >_{\sigma} x_1$. We  call this monomial order the decreasing revlexicographical order on $S.$

Any final lexsegment set $L^f(v)$, $v\in\Mc_d,$ satisfies the $\sigma$-exchange property for $<_{\sigma}$ as above, with respect to any monomial order $<$ on $R=K[\{T_w:w\in L^f(v)\}]$. In order to prove this claim, let $T_{u_1}\cdots T_{u_N}$ and $T_{v_1}\cdots T_{v_N}$ be two standard monomials of $J_{K[B]}$ with respect to $<$ such that $u_1\cdots u_N<_{\sigma} v_1\cdots v_N$, that is 
	\[u_1\cdots u_N>_{revlex} v_1\cdots v_N.
\]
Then there exists $1\leq q\leq n$ such that $\nu_i(u_1\cdots u_N)=\nu_{i}(v_1\cdots v_N)$ for all $i\geq q+1$ and $\nu_q(u_1\cdots u_N)<\nu_{q}(v_1\cdots v_N)$. Since $\deg(u_1\cdots u_N)=\deg(v_1\cdots v_N)$, we must have at least an index $j<q$ such that $\nu_j(u_1\cdots u_N)>\nu_j(v_1\cdots v_N)$. Let $1\leq \delta\leq N$ be such that $j\in\supp(u_\delta)$. Then the following conditions hold: $x_j>_{revlex}x_q$, that is $x_j<_{\sigma}x_q$ and $x_qu_{\delta}/x_j<_{lex}u_{\delta}$, whence $x_qu_{\delta}/x_j\in L^f(v)$.
\end{Example}

We also notice that, if we choose $<$ on $R$ to be the monomial order given in the previous section, that is the lexicographical order on the monomials $\{T_w\ :\ w\in L^f(v)\}$ induced by $T_{w_1}>T_{w_2}$ if $w_1>_{lex}w_2$, then $L^f(v)$ does not satisfy the $\ell$-exchange property with respect to $<$. For example, let $v=x_1x_3x_4\in K[x_1,x_2,x_3,x_4]$. Let $u_1=x_2^3$ and $v_1=x_1x_3x_4$, $u_1,v_1\in L^f(v).$ Then $\left(T_{u_1}\right)^2$ and $\left(T_{v_1}\right)^2$ are standard monomials with respect to $<$ on $R=K[\{T_w: w\in L^f(v)\}]$ and $u_1^2<_{lex}v_1^2$. In the $\ell$-exchange property, we have to take $q=1$. Since $\supp(u_1)=\{2\}$, we should have $x_1u_1/x_2=x_1x_2^2\in L^f(v)$, which is not possible.

Following closely the ideas from the last section in \cite{HHV}, we may prove a slight generalization of \cite[Theorem 5.1]{HHV}.

Let $I\subset S$ be a monomial ideal generated in degree $d$ and let $B=G(I)$ its minimal monomial generating set. Let $\Tc=S[\{T_u\}_{u\in B}]=K[x_1,\ldots,x_n,T_u:u\in B ]$ be the polynomial ring over $K$. $\Tc$ is bigraded by $\deg (x_i) =(1,0)$ for all $1\leq i\leq n$ and $\deg(T_u)=(0,1)$ for all $u\in B$.

Let $\Rc(I)=\bigoplus\limits_{j\geq0}I^jt^j=S[\{ut\}_{u\in B}]\subset S[t]$ be the Rees ring of $I$. $\Rc(I)$ is also naturally bigraded by $\deg(x_i)=(1,0)$ for $1\leq i\leq n$ and $\deg(ut)=(0,1)$ for all $u\in B$. There exists a canonical bigraded surjective $K$-algebra homomorphism $\varphi:\Tc\rightarrow\Rc(I)$ defined b $\varphi(x_i)=x_i$ for $1\leq i\leq n$ and $\varphi(T_u)=ut$ for all $u\in B$. Let $P_{\Rc(I)}:=\ker\varphi$ be the toric ideal of $\Rc(I)$. $P_{\Rc(I)}$ is bihomogeneous and generated by irreducible bihomogeneous binomials of $\Tc$. Let $<^{\#}$ be an arbitrary monomial order on $R$ and $<_{\sigma}$ be an arbitrary monomial order on $S$. By $<_{\sigma}^\#$ we will denote the product of these two orders which is a monomial order on $\Tc$. More precisely, for $mT_{u_1}\cdots T_{u_N}$, $m'T_{v_1}\cdots T_{v_N}$, monomials in $\Tc$, with $m,m'$ monomials in $S$, we have $mT_{u_1}\cdots T_{u_N}<_{\sigma}^{\#}m'T_{v_1}\cdots T_{v_N}$ if $m<_{\sigma}m'$ or $m=m'$ and $T_{u_1}\cdots T_{u_N}<^{\#}T_{v_1}\cdots T_{v_N}$. 

The following theorem generalizes  \cite[Theorem 5.1]{HHV}.

\begin{Theorem}\label{GB} Let $I\subset S$ be a monomial ideal generated in degree $d$, $B=G(I)$, $<^\#$ a monomial order on $R$ and $<_{\sigma}$ a monomial order on $S$. Let $G_{<^\#}(J_{K[B]})$ be the reduced Gr\"obner basis of the toric ideal $J_{K[B]}$ with respect to $<^\#$. Suppose that $B$ satisfies the $\sigma$-exchange property with respect to $<^{\#}$. Then the reduced Gr\"obner basis of the toric ideal $P_{\Rc(I)}$ with respect to $<_{\sigma}^\#$ consists of all binomials belonging to $G_{<^\#}(J_{K[B]})$ together with the binomials of the form
	\[x_iT_u-x_jT_v\in P_{\Rc(I)}
\]
where $x_j$ is the smallest variable with respect to $<_{\sigma}$ such that $x_i>_{\sigma}x_j$ and $x_iu/x_j\in~B$.
\end{Theorem}

\begin{proof} We closely follow the ideas from the proof of  \cite[Theorem 5.1]{HHV}.

We first show that the set
	\[G=G_{<^\#}(J_{K[B]})\cup\{x_iT_u-x_jT_v\in P_{\Rc(I)}\ : \ x_i>_{\sigma}x_j\}
\]
is a Gr\"obner basis of $P_{\Rc(I)}$ with respect to $<_{\sigma}^\#$.

Let $f\in P_{\Rc(I)}\subset \Tc$ be an irreducible binomial. If $\ini_{<_{\sigma}^\#}(f)\in R$, then $f\in P_{\Rc(I)}\cap R=J_{K[B]}$, hence there is a binomial belonging to $G_{<^\#}(J_{K[B]})$ which divides $\ini_{<_{\sigma}^\#}(f)$. 

Let $\ini_{<_{\sigma}^\#}(f)\notin R$, that is, we may write
	\[f=x_{i_1}\cdots x_{i_t}T_{u_1}\cdots T_{u_N}-x_{j_1}\cdots x_{j_t}T_{v_1}\cdots T_{v_N}
\]
with $\{i_1,\ldots,i_t\}\cap\{j_1,\ldots,j_t\}=\emptyset$ and where we assume that $x_{i_1}\geq_{\sigma}\cdots\geq_{\sigma}x_{i_t}$ and $x_{j_1}\geq_{\sigma}\ldots\geq_{\sigma} x_{j_t}$. By successively reductions modulo the binomials from $G_{<^\#}(J_{K[B]})$ we may assume that $T_{u_1}\cdots T_{u_N}$ and $T_{v_1}\cdots T_{v_N}$ are standard monomials with respect to $<^\#$. Let $\ini_{<_{\sigma}^\#}(f)=x_{i_1}\cdots x_{i_t}T_{u_1}\cdots T_{u_N}$. Then $x_{i_1}\cdots x_{i_t}>_{\sigma}x_{j_1}\cdots x_{j_t}$. By using the equality 
\[x_{i_1}\cdots x_{i_t}u_1\cdots u_N=x_{j_1}\cdots x_{j_t}v_1\cdots v_N,
\]
we obtain $u_1\cdots u_N<_{\sigma}v_1\cdots v_N$, $\nu_{i_s}(u_1\cdots u_N)<\nu_{i_s}(v_1\cdots v_N)$ for $1\leq s\leq t,$ and $\nu_k(u_1\cdots u_N)\geq\nu_k(v_1\cdots v_N)$ for all $k\notin\{i_1,\ldots, i_t\}$. Since $B$ satisfies the $\sigma$-exchange property, we have that there exist $1\leq \delta\leq N$, $j\in\supp(u_{\delta})$ and $q\in\supp(v_1\cdots v_N)$ such that $\nu_q(u_1\cdots u_N)<\nu_q(v_1\cdots v_N)$, $x_j<_{\sigma}x_q$, and $x_qu_{\delta}/x_j\in B$. 

The first above condition on $q$ shows that $q=i_s$, for some $1\leq s\leq t$. Therefore we have $x_{i_s}u_{\delta}=x_{j}v$ for some $v\in B$ and the proof of our claim is finished.

To end the proof, let us take some binomial $x_iT_u-x_jT_v$, where $u,v\in B$, $x_iu=x_jv$ and $x_j<_{\sigma}x_i$ is the smallest variable with respect to $<_{\sigma}$ such that $x_iu/x_j\in B$.
Assume that $x_jT_v$ is not reduced, hence there exists some binomial $x_jT_v-x_{l}T_w$ with $x_l<_{\sigma}x_j,$ which belongs to   $ P_{\Rc(I)}$. Then $x_iT_u-x_lT_w\in P_{\Rc(I)}$ and $x_{l}<_{\sigma}x_j<_{\sigma}x_i$, a contradiction. 
\end{proof}

\begin{Corollary}\label{corfiber} Let $I\subset S$ be a monomial ideal generated in degree $d$ and $B=G(I)$. Let $<^\#$ be a monomial order on $R$ and $<_{\sigma}$ a monomial order on $S$. If $B$ satisfies the $\sigma$-exchange property with respect to $<^\#$, then $I$ is of fiber type.
\end{Corollary}

We recall (see \cite{HHV}) that an ideal $I\subset S$ is called \textit{of fiber type} if the fiber relations together with the relations of the symmetric algebra of $I$ generate all the relations of the Rees algebra of $I.$

The above corollary may be used to find equigenerated monomial ideals of fiber type. 
Let $<_{\sigma}$ be an arbitrary graded monomial order on $S$, $u\in\Mc_d$ and $I=(L^i_{<_{\sigma}}(u))$, where $L^i_{<_{\sigma}}(u)=\{w\in\Mc_d:w>_{\sigma}u\}$. Then it is easily seen that $(L^i_{<_{\sigma}}(u))$ satisfies the $\sigma$-exchange property for any monomial order on $R=K[\{T_w:w\in L^i_{<_{\sigma}}(u)\}]$, hence $I$ is of fiber type.

We prove now a significant property of the monomial ideals whose minimal monomial generating system satisfies a $\sigma$-exchange property.

\begin{Theorem}\label{powersexchange} Let $I\subset S$ be a monomial ideal generated in degree $d$ and $B=G(I)$. Let $<^\#$ be a monomial order on $R=K[\{T_u:u\in B\}]$ and $<_{\sigma}$ a monomial order on $S$. If $B$ satisfies the $\sigma$-exchange property with respect to $<^{\#}$, then $I^N$ has linear quotients with respect to $>_{\sigma}$ for  $N\geq1$.
\end{Theorem}

\begin{proof} Let $G\left(I^N\right)=\{w_1,\cdots,w_r\}$, where $w_1>_{\sigma}\cdots>_{\sigma}w_r$ and let $T_{w_1}, \ldots,T_{w_r}$ be standard monomials of $J_{K[B]}$ with respect to $<^\#$. Let $1\leq j<i\leq r$ be two integers and assume that $w_j=v_1\cdots v_N$ and $w_i=u_1\cdots u_N$ for $u_1,\ldots,u_N,v_1,\ldots,v_N\in G(I)$, $u_1\geq_{\sigma}\cdots\geq_{\sigma}u_N$, $v_1\geq_{\sigma}\cdots\geq_{\sigma}v_N$. We have to prove that there exist $1\leq k<i$ and $1\leq q\leq n$ such that 
	\[\frac{w_k}{\gcd(w_k,w_i)}=x_q\mbox{ and } x_q\mid\frac{w_j}{\gcd(w_j,w_i)} .
\]
Since $w_j>_{\sigma}w_i$, by using the $\sigma$-exchange property of $B$, there exist $1\leq \delta\leq N$, $l\in\supp(u_{\delta})$, and $q\in\supp(v_1\cdots v_N)$ such that $\nu_q(u_1\cdots u_N)<\nu_q(v_1\cdots v_N)$, $x_l<_{\sigma}x_q$, and $x_qu_{\delta}/x_l\in B$. Let 
\[w_k=u_1\cdots u_{\delta-1}\frac{x_qu_{\delta}}{x_l}u_{\delta+1}\cdots u_N=\frac{x_qw_i}{x_l}.
\]
Then $w_{k}$ satisfies the required conditions.
\end{proof}

In the sequel we show that the lexsegment ideals with a linear resolution which are not completely  satisfies an exchange property.

We first recall the following
\begin{Theorem}[\cite{ADH}] Let $I=(L(u,v))$ be a lexsegment ideal with $x_1\mid u$ and $x_1\nmid v$ which is not a completely lexsegment ideal. Then $I$ has a linear resolution if and only if $u$ and $v$ have the following form:
	\[u=x_1x_{l+1}^{a_{l+1}}\cdots x_n^{a_n}\ \mbox{and } v=x_lx_{n}^{d-1}
\]
for some $l$, $2\leq l\leq n-1$.
\end{Theorem}

\begin{Theorem} Let $<_{\sigma}$  be the decreasing revlexicographical order on $S$ and $I=(L(u,v))$ a lexsegment ideal with linear resolution which is not a completely lexsegment ideal. Then $L(u,v)$ satisfies the $\sigma$-exchange property with respect to any monomial order on $R=K[\{T_w:w\in L(u,v)\}]$.
\end{Theorem}

\begin{proof} Let $u=x_1x_{l+1}^{a_{l+1}}\cdots x_n^{a_n}$, $v=x_lx_n^{d-1}$ for some $2\leq l\leq n-1$. Let us assume that there exists a monomial order $<$ on $R$ such that $L(u,v)$ does not satisfy the $\sigma$-exchange property with respect to $<$. Then there exist two standard monomials $T_{u_1}\cdots T_{u_N}$ and $T_{v_1}\cdots T_{v_N}$ such that $u_1\cdots u_N>_{revlex}v_1\cdots v_N$ and with the property that for all $1\leq \delta\leq N$, $j\in\supp(u_{\delta})$ and $q\in\supp(v_1\cdots v_N)$ such that $\nu_q(u_1\cdots u_N)<\nu_{q}(v_1\cdots v_N)$ and $x_j>_{revlex}x_q$, we have $x_qu_{\delta}/x_j\notin L(u,v)$. Since $u_1\cdots u_N>_{revlex}v_1\cdots v_N$ there exists some $q$, $1\leq q\leq n$, such that
	\[\nu_i(u_1\cdots u_N)=\nu_i(v_1\cdots v_N) \mbox{ for all }i\geq q+1
\]
and $\nu_q(u_1\cdots u_N)<\nu_{q}(v_1\cdots v_N)$. Since $\deg(u_1\cdots u_N)=\deg(v_1\cdots v_N)$ there exists some $s<q$ such that $\nu_s(u_1\cdots u_N)>\nu_{s}(v_1\cdots v_N)$. Let $u_{\delta}$ be such that $s\in\supp(u_{\delta})$. By our assumption, we must have $x_qu_{\delta}/x_s<_{lex}v$, that is $x_qu_{\delta}/x_s\leq_{lex}x_{l+1}^d$. This implies, in particular, that $q\geq l+1$, and that for all $\delta$, $1\leq \delta\leq N$, there exists a unique $j_{\delta}\leq l$ such that $u_{\delta}=x_{j_{\delta}}w_{\delta}$ where $\min(w_{\delta})\geq l+1$.

Therefore we have $u_{1}\cdots u_{N}=x_{j_1}\cdots x_{j_N}x_{t}^{a_t}\cdots x_n^{a_n}$, where $j_1,\cdots,j_N< q$ and $t\geq q$. We have
	\[a_t+\cdots+a_n=\deg(x_{t}^{a_t}\cdots x_{n}^{a_n})=Nd-N=N(d-1).
\]
Let $v_1\cdots v_N=x_1^{b_1}\cdots x_n^{b_n}$. By hypothesis, we have $a_q<b_q$ and $a_i=b_i$ for all $i\geq q+1$. Since each monomial $v_{\gamma}\in L(u,v)$ it is divisible by some variable $x_i$ with $i\leq l<q$, we have $b_1+\cdots+b_{q-1}\geq N$.
Then we have
	\[Nd=b_1+\cdots+b_{q-1}+b_q+\cdots+b_n>b_1+\cdots+b_{q-1}+a_q+\cdots+a_t+\cdots+a_n\geq
	\]
	\[\geq N+N(d-1)=Nd,
\]
a contradiction.
\end{proof}

\begin{Corollary}\label{cornoncomp} All powers of a lexsegment ideal with a linear resolution which is not a completely lexsegment ideal have linear quotients with respect to the increasing revlexicographic order.
\end{Corollary}

\begin{Corollary} Any lexsegment ideal with a linear resolution which is not a completely lexsegment ideal is of fiber type.
\end{Corollary}

\begin{Corollary}\label{lexKoszul} Let $I=(L(u,v))$ be a lexsegment ideal with a linear resolution which is not a completely lexsegment ideal. Then the Rees algebra $\Rc(I)$ is Koszul.
\end{Corollary}

\begin{proof} Let $<^\#$ be the lexicographical monomial order on $R=K[\{T_w:w\in L(u,v)\}]$ induced by $T_{w_1}>T_{w_2}$ if $w_1>_{lex}w_2$ and $<_{\sigma}$ be the decreasing revlexicographic order on $S$. By Theorem \ref{GB}, the reduced Gr\"obner basis of $P_{\Rc(I)}$ with respect to the product order $<_{\sigma}^{\#}$ on $\Tc$ is formed by the binomials from $G_{<^\#}\left(J_{K[L(u,v)]}\right)$, the reduced Gr\"obner basis of $J_{K[L(u,v)]}$, and by the binomials of the form
	\[x_iT_{u'}-x_jT_{v'},
\]
where $x_i>_{\sigma}x_{j}$, $x_iu'=x_jv'$ and $j $ is the smallest integer with $x_iu'/x_j\in L(u,v)$. Since $G_{<^\#}\left(J_{K[L(u,v)]}\right)$ is quadratic (\cite[Proposition 2.13]{D}), the statement follows. 
\end{proof}

\end{document}